\documentclass[a4paper,9pt]{article}
\usepackage{amsfonts}
\usepackage{amssymb,amsthm,comment}
\usepackage{amsmath,subfig,curves,hyperref}
\usepackage[usenames]{color}

\newcommand{\R}{\mathbb{R}}
\newcommand{\N}{\mathbb{N}}

\newcommand{\Z}{\mathbb{Z}}

\newtheorem{theorem}{Theorem}

\newtheorem{corollary}[theorem]{Corollary}
\newtheorem{example}[theorem]{Example}

\newcommand{\be}{\begin{equation}}
\newcommand{\ee}{\end{equation}}
\newcommand{\bee}{\begin{equation*}}
\newcommand{\eee}{\end{equation*}}
\newcommand{\bp}{\begin{proof}}
\newcommand{\ep}{\end{proof}}

\newtheorem{prop}[theorem]{Proposition}

\newcommand{\red}{\color{red}}

\begin{document}
\title{Isospectrality and heat content}
\author{{M. van den Berg, E. B. Dryden \thanks{Emily Dryden was partially supported by a grant from the Simons Foundation (210445 to Emily B.
Dryden). She also thanks the School of Mathematics at Trinity College Dublin for their hospitality.}, T. Kappeler \thanks{Partially supported by the Swiss National Science Foundation.}} \\ \\
School of Mathematics, University of Bristol\\
University Walk, Bristol BS8 1TW, UK\\
\texttt{M.vandenBerg@bris.ac.uk}\\ \\
Department of Mathematics, Bucknell University\\
 Lewisburg, PA 17837, USA\\
\texttt{emily.dryden@bucknell.edu}\\ \\
Institut f\"ur Mathematik, Universit\"at Z\"urich\\
 Winterthurerstrasse 190,
CH-8057 Z\"urich, Switzerland\\
\texttt{thomas.kappeler@math.uzh.ch}}


\date{27 January 2014}\maketitle
\begin{abstract}\noindent We present examples of isospectral operators that do not have the same heat content. Several of these examples are planar polygons
that are isospectral for the Laplace operator with Dirichlet
boundary conditions. These include examples with infinitely many
components. Other planar examples have mixed Dirichlet and Neumann
boundary conditions. We also consider Schr\"{o}dinger operators
acting in $L^2[0,1]$ with Dirichlet boundary conditions, and show
that an abundance of isospectral deformations do not preserve the
heat content.
\end{abstract}
\vskip 1truecm \noindent \ \ \ \ \ \ \ \  { Mathematics Subject
Classification (2010): 58J53; 58J50; 35P10.
\begin{center} \textbf{Keywords}: Isospectral, heat content, isoheat, polygons, Schr\"odinger operators.\\
\end{center} \mbox{}\newpage
\section{Introduction}
\label{sec1}

Let $(M,g)$ be a compact $m$-dimensional Riemannian manifold, not
necessarily connected and with a boundary $\partial M$. Let
$-\Delta_M$ be the associated Dirichlet Laplace-Beltrami operator
acting in $L^2(M,dx)$, where $dx$ is the volume measure on $M$
induced by $g$. The spectrum of $-\Delta_M$ is discrete and we
denote the eigenvalues by $\lambda_1(M)\le \lambda_2(M)\le
\lambda_3(M)\le \dots \ ,$ where each eigenvalue is repeated
according to its multiplicity. See \cite{PG} for details. Two
Riemannian manifolds $M_1$ and $M_2$ are isospectral if
$\lambda_j(M_1)=\lambda_j(M_2),\ j=1,2,\dots $. We refer to
\cite{GPS} and the references therein for a survey and details on
isospectrality. We denote the heat trace for $M$ by
\begin{equation*}
Z_M(t)= \textup{Tr}_{L^2}(e^{t\Delta_M})=\sum_{j=1}^{\infty}e^{-t\lambda_j(M)},\ t>0.
\end{equation*}
The heat trace for $M$ is a spectral invariant, and we note that
$Z_M(t)$ determines the spectrum, i.e., the eigenvalues of
$-\Delta_M$ together with their multiplicities. Moreover its
asymptotic behaviour as $t\downarrow 0$ contains geometric
information on $M$. It has been shown that if $\partial M$ is
smooth then there exists an asymptotic series such that for any
$I\in \N$
\begin{equation}\label{e2}
Z_M(t)=\sum_{i=0}^{I}a_i(M)t^{(i-m)/2}+O(t^{(I+1-m)/2}),\
t\downarrow 0,
\end{equation}
where the $a_i(M)$ are locally computable spectral invariants. See
\cite{PG} for further details.

Of much interest is the heat content of $M$. Let $u:M\times
[0,\infty)\rightarrow \R$ be the unique solution of
\begin{equation}\label{e3}
   \Delta_M u= \frac{\partial u}{\partial t},\ \quad  t>0,
\end{equation}
with initial condition
\begin{equation}\label{e4}
    \lim_{t\downarrow 0}u(\cdot;t)=1\ \textup{in} \ L^2(M),
\end{equation}
and Dirichlet boundary conditions
\begin{equation}\label{e5}
    u(x;t)=0,\ \  x\in \partial M,\ t>0.
\end{equation}
The heat content $Q_M(t),t\ge 0$, is defined by
\begin{equation*}
   Q_M(t)=\int_M u (x;t) dx
\end{equation*}
(see \cite{vdB1,PG} and the references therein). It has been shown
that if $\partial M$ is smooth then there exists an asymptotic
series such that for any $I\in \N$
\begin{equation}\label{e7}
Q_M(t)=\sum_{i=0}^{I}b_i(M)t^{i/2}+O(t^{(I+1)/2}),\ t\downarrow 0,
\end{equation}
where the $b_i(M)$ are locally computable heat content invariants.
We refer to \cite{vdBGKK} for the existence and to \cite{vdB1} for
the calculation of the first few coefficients.

The solution $u$ of \eqref{e3},\eqref{e4} and \eqref{e5} can be
expressed in terms of the spectral resolution
$\left\{\phi_{j,M},\lambda_j(M)\right\}$ of $-\Delta_M$,
where $\{\phi_{1,M},\phi_{2,M},\phi_{3,M},\dots\}$ denotes an
orthonormal basis of eigenfunctions corresponding to
$\lambda_1(M)\le \lambda_2(M)\le \lambda_3(M)\le \dots$. Let
$p_M(x,y;t), x\in M, y\in M, t>0$ denote
 the Dirichlet heat kernel for $M$. Then
\begin{equation*}
p_M(x,y;t)=\sum_{j=1}^{\infty}e^{-t\lambda_j(M)}\phi_{j,M}(x)\phi_{j,M}(y).
\end{equation*}
We have that
\begin{equation*}
u(x;t)=\int_Mp_M(x,y;t)dy=\sum_{j=1}^{\infty}e^{-t\lambda_j(M)}\left(\int_M\phi_{j,M}(y)dy\right)\phi_{j,M}(x).
\end{equation*}
By definition of the heat content,
\begin{equation}\label{e10}
Q_M(t)=\sum_{j=1}^{\infty}e^{-t\lambda_j(M)}\left(\int_M\phi_{j,M}(x)dx\right)^2.
\end{equation}

We say that $M_1$ and $M_2$ are {\it isoheat} if for all $t>0$,
$Q_{M_1}(t)=Q_{M_2}(t)$.

The number $\int_M\phi_{j,M}(y)dy$ is just the $j^{th}$
Fourier coefficient of the initial datum, which is the constant
function $1$ on $M$.  It follows from \eqref{e10} that if all
eigenvalues are simple, and if this Fourier coefficient is nonzero
for all $j\in \N$, then $Q_M(t)$ determines the spectrum. Note
that the invariants one can read off from \eqref{e10} are not
local.

 From formula \eqref{e10} one cannot decide if
$Q_M(t)$ is a spectral invariant like $Z_M(t)$. Nevertheless it
shares some of its features, like the existence of an asymptotic
expansion for $t\downarrow 0$. The aim of this paper is to explore
the relation between isospectrality and the property of being
isoheat.

It is known that isoheat does not imply isospectral \cite{PG2}.
Below we give two new examples demonstrating this, but first we
make a general observation.  If $M$ does not have a boundary, we
have that $\lambda_1(M)=0$ with $\phi_{1,M}=|M|^{-1/2}$ up to a
sign, where $|M|=\int_M 1dx$. All but the first Fourier
coefficients equal $0$, and so $Q_M(t)=|M|$.
\begin{prop}\label{cor1}Let $M_1$ and  $M_2$ be manifolds
without boundary and with the same volume.  Then $M_1$ and  $M_2$
are isoheat.
\end{prop}

\begin{example}\label{exa2} {\rm Let $ M_1$ and  $M_2$ be compact Riemannian
manifolds without boundary such that $|M_1|=| M_2|$, $ M_1$ is
connected, and $ M_2$ is not connected. $ M_1$ and $ M_2$ are not
isospectral since $ M_1$ has a first eigenvalue $0$ with
multiplicity $1$ whereas $ M_2$ has a first eigenvalue $0$ with
multiplicity at least $2$. By Proposition \ref{cor1}, $ M_1$ and $
M_2$ are isoheat. If $ M_3$ is a compact Riemannian manifold with
smooth boundary then $ M_1\times  M_3$ and $ M_2\times  M_3$ are
isoheat manifolds with boundary which are not isospectral. Also
note that if $ M_1$ and $M_2$ are isoheat and $ M_3$ and $ M_4$
are isoheat then $ M_1\times  M_3$ and $ M_2\times
 M_4$ are isoheat.}
\end{example}

Our second example consists of planar polygons.  Both the heat
trace and heat content are well defined and admit expansions
similar to \eqref{e2} and \eqref{e7} for these planar domains.
Recall that Neumann boundary conditions are given by
\begin{equation}\label{Neumann}
    \frac{\partial u}{\partial n} (x;t)=0,\ \  x\in \partial M,\ t>0,
\end{equation}
where $n$ is an outward unit normal vector field on $\partial M$.

 \begin{example}\label{exa3} {\rm Let
$ A$ be a rectangle with edges of lengths $1$ and $2$ and with
Dirichlet boundary conditions. Let $B$ be the disjoint union of
two squares of edge length $1$, where each square has Dirichlet
boundary conditions on three edges and Neumann boundary conditions
on the remaining edge. The heat flow in $ A$ is symmetric with
respect to the axes of symmetry, so in particular satisfies
Neumann boundary conditions along the short axis. Thus the heat
content of $ A$ is the same as the heat content of $ B$. We
observe that $\lambda_1( A)=\frac{5\pi^2}{4}$ has multiplicity $1$
and $\lambda_1( B)=\frac{5\pi^2}{4}$ has multiplicity $2$. Hence $
A$ and $ B$ are not isospectral.}
\end{example}

It was shown in \cite{PG2} that a Sunada construction \cite{S}
involving finite coverings will not produce isospectral manifolds
that are not isoheat. However, the general question of whether
there exist isospectral manifolds with different heat contents was
hitherto unanswered.  Note that the question has been resolved in
\cite{MMgraphs} for weighted graphs: McDonald and Meyers construct
weighted graphs associated to the principal example in
\cite{BCDS}, which is a simplified version of the drums of
\cite{GWW}.  These graphs are planar, isospectral, and
non-isometric; using the heat operator naturally associated to
such a weighted graph, the authors calculate the heat content of
both graphs.  They show that the fifth coefficient $b_4$ in the
small-time asymptotic expansion of the heat content differs for
the two graphs, hence they are not isoheat.

In Section \ref{secplanarexas}, we present a variety of examples,
where explicit computations are possible, to show that isospectral
does not imply isoheat in the setting of open sets in Euclidean
space. Our examples consist of planar domains, and mostly of
planar polygons. Some have all Dirichlet boundary conditions and
others have mixed boundary conditions.  We present examples in
which both sets of the isospectral pair consist of the same number
of connected components, and an example in which one set is
connected and the other is disconnected. To show that our sets are
not isoheat, we use two approaches: (i) we either show that the
coefficients in the small-time asymptotic expansion of the heat
content differ for the two sets, or (ii) we consider the
large-time behaviour of the heat content. In particular, for the
example consisting of disconnected sets with infinitely many
components in each, we use the large-time behaviour.  The
construction of these infinite disconnected sets involves a family
of self-similar polygons, and it is interesting to obtain the
first few terms in the expansion for $t \downarrow 0$ of the heat
trace and heat content of such a ``fractal'' family. See Theorem
\ref{the2} and Theorem \ref{the4} in the Appendix.

In Section \ref{sec3}} we consider Schr\"odinger operators
$L_q=-\frac{d^2}{dx^2}+q$ acting in $L^2[0,1]$ with Dirichlet
boundary conditions at $0$ and at $1$. We assume that $q\in
L^2[0,1]$. $L_q$ has discrete spectrum and its eigenvalues are
simple: $\lambda_1(q)<\lambda_2(q)<\lambda_3(q)<\dots $ (e.g.,
\cite[Chap. 2]{PT}). As in the case with $-\Delta_M$ one defines
the heat trace $Z_q(t)$ and the heat content $Q_q(t)$ of $L_q$.
Two potentials $q_1$ and $q_2$ in $L^2[0,1]$ are isospectral if $\lambda_j(q_1)=\lambda_j(q_2)$ for
all $j\in \N$.

We say that $q_1$ and $q_2$ in $L^2[0,1]$ are
{\it isoheat} if for all $t>0$, $Q_{q_1}(t)=Q_{q_2}(t)$.

We wish to explore the relationship between isospectrality and the
property of being isoheat using the fact that these model
operators have simple eigenvalues. First we note that as in the
case of $-\Delta_M$, the heat content of $L_q$ has a
representation of the form

\begin{equation*}
Q_q(t)=\sum_{j=1}^{\infty}e^{-t\lambda_j(q)}\left(\int_0^1\phi_{j,q}(x)dx\right)^2,
\end{equation*}
where for any $j\in \N$, $\phi_{j,q}$ denotes the eigenfunction
corresponding to $\lambda_j(q)$ normalized by
$\int_0^1\phi_{j,q}(x)^2dx=1$ and $\phi_{j,q}'(0)>0$. Note that
$\phi_{j,q}\in H^2[0,1]$ and hence by the Sobolev embedding
theorem,  $\phi_{j,q}\in C^1[0,1]$. According to \cite{PG} for any
smooth potential $q$, $Q_q(t)$ admits an asymptotic expansion of
the form

\begin{equation}\label{e154}
Q_q(t)=\sum_{i=0}^Ib_i(q)t^{i/2}+O(t^{(I+1)/2}),\ \
t\downarrow 0,
\end{equation}
where $I\in \N$ and $b_0(q)=1, b_1(q)=-4\pi^{-1/2}$ and
$b_2(q)=-\int_0^1q(x)dx$. In particular it follows that the mean
of $q$ is an invariant of the heat content. Since $b_1(q)\ne 0$ we
have that the set $J_q=\{j\in \N: \int_0^1\phi_{j,q}\ne 0\}$ is
infinite. We remark that from the asymptotics
$\lambda_j(q)=\pi^2j^2-b_2(q)+o(1), j \rightarrow \infty$ \cite[p.
35]{PT}, it follows that the mean is also a spectral invariant
(see also \cite{PG}).

There is a natural pair of isospectral isoheat potentials.
Consider for any $q\in L^2[0,1]$ the reflected potential
$q_*(x)=q(1-x), 0\le x\le 1$. According to \cite[p. 52]{PT}, $q_*$
and $q$ are isospectral. Furthermore, the eigenfunctions
corresponding to $q_*$, denoted $\phi_{j,q_*}(x)$, are related to
the eigenfunctions corresponding to $q$ by
$\phi_{j,q_*}(x)=(-1)^{j+1}\phi_{j,q}(x)$ \cite[p. 42]{PT}. As a
consequence
$\int_0^1\phi_{j,q_*}(x)dx=(-1)^{j+1}\int_0^1\phi_{j,q}(x)dx$,
implying that $q$ and $q_*$ are also isoheat. In Section
\ref{sec3}, we give an abundance of isospectral deformations that
are not isoheat.


\section{Heat content and planar domains}\label{secplanarexas}

As mentioned in the Introduction, both the heat trace
and the heat content are well defined and admit expansions similar
to \eqref{e2} and \eqref{e7} in the case of planar polygons with
associated Dirichlet Laplace operator. These are bounded regions
in $\R^2$, not necessarily connected nor simply connected, such
that each component of the boundary consists of vertices that are
connected by straight line segments. It was shown in \cite{vdB2}
that if $P$ is a polygon in $\R^2$ with vertices $V_1, V_2, \dots,
V_n$ with corresponding inward pointing angles
$\gamma_1,\gamma_2,\dots,\gamma_n$ then there exist constants
$d_1(P),d_2(P)$ depending on $P$ such that
\begin{equation}\label{e11}
\Biggl\lvert Z_P(t)-\frac{|P|}{4\pi t}+\frac{|\partial P|}{8(\pi
t)^{1/2}}-\sum_{i=1}^n\frac{\pi^2-\gamma_i^2}{24\pi
\gamma_i}\Biggr\rvert \le d_1(P)e^{-d_2(P)/t},\ t>0,
\end{equation}
where $|P|$ is the area of $P$ and $|\partial P|$ is the sum of
the lengths of the components of the boundary. For example if
$\partial P$ is connected then $|\partial
P|=\sum_{i=1}^{n-1}|V_i-V_{i+1}|+|V_n-V_1|$.

For the heat content it was shown in \cite{vdB3} that there exist
constants $d_3(P),d_4(P)$ depending on $P$ such that
\begin{equation}\label{e12}
\Biggl\lvert Q_P(t) - |P| + \frac{2 t^{1/2}}{\pi^{1/2}}|\partial
P| - t \sum_{i=1}^n c(\gamma_i)\Biggr\rvert\le
d_3(P)e^{-d_4(P)/t},\ t\ge 0
\end{equation}
where
\begin{equation*}
c(\gamma) = \int_0^{\infty} \frac{4 \sinh ((\pi -
\gamma)x)}{(\sinh (\pi x))(\cosh (\gamma x))} \ dx.
\end{equation*}

It follows from \eqref{e12} that we have an expansion of the form
\eqref{e7}, where $b_0(P)=|P|,\ b_1(P)=-2\pi^{-1/2}|\partial P|,\
b_2(P)=\sum_{i=1}^n c(\gamma_i),\ b_3(P)=b_4(P)=\dots=0$. Both the
proofs of \eqref{e11} and \eqref{e12} rely on detailed
calculations using the representation of the heat kernel for the
infinite wedge as a Kontorovich-Lebedev transform. It was reported
in \cite{MKS} that D. B. Ray obtained the contribution of the
vertices $\sum_{i=1}^n\frac{\pi^2-\gamma_i^2}{24\pi \gamma_i}$ in
\eqref{e11} using this transform.

We now give several examples of isospectral planar domains
that are not isoheat.  In the first example, one domain is
connected and has only Dirichlet boundary conditions, but its
isospectral partner is disconnected and has mixed boundary
conditions.
\begin{example}\label{exa4}
{\rm Let $A$ be a rectangle with edges of lengths $1$ and $2$ and
with Dirichlet boundary conditions on all edges. Let $B$ be the
disjoint union of two unit squares, one with all Dirichlet
boundary and one with Dirichlet boundary conditions on three edges
and Neumann boundary conditions on the remaining edge. The even
eigenfunctions in $ A$ satisfy Neumann boundary conditions along
the short axis of symmetry. There is a one-to-one correspondence
between these even eigenfunctions in $ A$ and eigenfunctions
 of the square in $B$ with mixed boundary
conditions. A similar situation holds for the odd eigenfunctions
in $ A$ and the eigenfunctions of the square in $ B$ with
Dirichlet boundary conditions. Thus $ A$ and $ B$ are isospectral.
However, note that $|\partial
 A|=6$. Hence by \eqref{e12}, the coefficient of $t^{1/2}$ in the
expansion for $Q_{ A}(t)$ is $-12\pi ^{-1/2}$. The square in $B$
with Dirichlet boundary conditions has perimeter $4$, and so this
gives a contribution $-8\pi ^{-1/2}$ to the coefficient of
$t^{1/2}$. Next consider the remaining component in $ B$. By
reflecting with respect to the Neumann edge, and deleting this
edge subsequently we obtain a rectangle with perimeter $6$. The
coefficient of $t^{1/2}$ for the rectangle is $-12\pi ^{-1/2}$. By
symmetry there is no heat flow across the axes of symmetry, and we
conclude that the coefficient of $t^{1/2}$ for the remaining
component in $B$ equals $-6\pi ^{-1/2}$. Hence the coefficients of
$t^{1/2}$ are $-12\pi ^{-1/2}$ and $-14\pi ^{-1/2}$ for $ A$ and $
B$ respectively. So $ A$ and $ B$ are not isoheat.}
\end{example}

The domains in the following two examples both consist of one
component and have mixed boundary conditions.

\begin{example}\label{exa5} {\rm The half-disks with mixed boundary
conditions as shown in Figure \ref{fig:half-disks} are isospectral
\cite{JLNP}. The boundary of $ A$ is endowed with Dirichlet
boundary conditions exactly where the boundary of $B$ is endowed
with Neumann boundary conditions, and vice versa. To see that $ A$
and $ B$ are not isoheat we consider the difference $Q_{ A}(t)-Q_{
B}(t)$. To this end we place vertices on the boundaries of the
half-disks at the points where the equator meets the circle
(``endpoint vertices'') and anywhere else the boundary conditions
change.  We see that for each possible pairing of boundary
conditions at a vertex, there is a bijection between the
non-endpoint vertices in $ A$ and the non-endpoint vertices in $
B$. Moreover, each of $ A$ and $ B$ have one endpoint vertex where
a Dirichlet and a Neumann edge meet. This leaves one endpoint
vertex $V_{ A}$ in $ A$ where two Neumann edges meet, and one
endpoint vertex $V_{ B}$ in $ B$ where two Dirichlet edges meet.
Since $V_{ A}$ has a boundary neighbourhood with Neumann boundary
conditions there is no contribution to $b_2( A)$ from this vertex.
On the other hand the geometry of the boundary near $V_{ B}$ is in
first approximation that of a wedge with angle $\pi/2$. This
suggests that an angle contribution $c(\pi/2)t=4t/\pi$ will show
up in $Q_{ B}(t)$, and that $Q_{ A}(t)-Q_{ B}(t)=-4t/\pi+o(t)$ so
that $ A$ and $ B$ are not isoheat. This can be made rigorous
using the Brownian motion tools from \cite{vdB3,vdBLG}.}
\end{example}

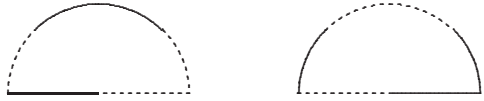
\begin{figure}
\centering \setlength{\unitlength}{4mm}
\begin{picture}(4,4)
\put(0,0){\line(-1,0){3}} \curvedashes[2mm]{0.25,0.25}
\put(0,0){\curve(0,0,1,0,2,0,3,0)} \curvedashes[2mm]{0.25,0.25}
\put(0,0){\arc(3,0){45}} \curvedashes[2mm]{0,1}
\put(0,0){\arc(2.1,2.1){90}} \curvedashes[2mm]{0.25,0.25}
\put(0,0){\arc(-2.1,2.1){45}}
\end{picture}
\qquad \qquad \qquad
\begin{picture}(4,4)
\put(0,0){\curve(0,0,1,0,2,0,3,0)} \curvedashes[2mm]{0.25,0.25}
\put(0,0){\curve(0,0,-1,0,-2,0,-3,0)} \curvedashes[2mm]{0,1}
\put(0,0){\arc(3,0){45}} \curvedashes[2mm]{0.25,0.25}
\put(0,0){\arc(2.1,2.1){90}} \curvedashes[2mm]{0,1}
\put(0,0){\arc(-2.1,2.1){45}}
\end{picture}
\caption{Isospectral half-disks $ A$ and $ B$ with mixed boundary
conditions, where solid lines represent Dirichlet boundary
conditions and dotted lines represent Neumann boundary conditions
\cite{JLNP}} \label{fig:half-disks}
\end{figure}

\begin{example}\label{exa6} {\rm Let $E$ be a square with area
$1$ and let $F$ be a right isosceles triangle with area $1$.
Suppose that three of the edges of $E$ have Dirichlet boundary
conditions and the remaining edge has Neumann boundary conditions,
and that one of the edges of $F$ with length $\sqrt 2$ has Neumann
boundary conditions while the remaining two edges have Dirichlet
boundary conditions. Then Levitin, Parnovski and Polterovich
\cite{LPP} have shown that $E$ and $F$ are isospectral. As shown
in Example \ref{exa4} the coefficient of $t^{1/2}$ in $Q_E(t)$
equals $-6\pi^{-1/2}$. Using this reflection argument for $F$, and
using \eqref{e12} for the double of $F$ we obtain that the
coefficient of $t^{1/2}$ in $Q_F(t)$ equals
$-2(2+2^{1/2})\pi^{-1/2}$. Hence $E$ and $F$ are not isoheat.}
\end{example}

In the remaining examples, we take a different approach to that
taken above to show that these isospectral sets are not isoheat.
Instead of using the coefficients in the small-time asymptotic
expansion of the heat content, we consider the behaviour for $t
\rightarrow \infty$.  The heat content function behaves like its
leading term $e^{-t \lambda_1} (\int \phi_1)^2$, so we show that
$(\int \phi_1)^2$ differs for our isospectral domains. The reader
can easily check that the coefficients in the small-time expansion
actually match for the isospectral sets considered below.

Chapman \cite{C} constructed several examples of isospectral,
non-isometric drums. We call a disconnected drum a band. We denote
the two-piece band consisting of the square of edge length $1$ and
the right isosceles triangle with area $2$ by $ A$, and the pair
consisting of the rectangle of edge lengths $1$ and $2$ and the
right isosceles triangle with area $1$ by $ B$.

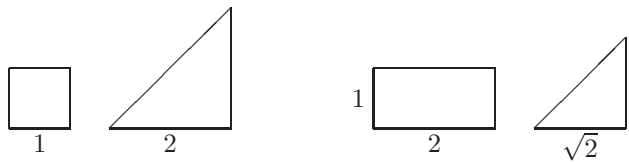
\begin{figure}[h]
\centering 
{ \setlength{\unitlength}{4mm}
\begin{picture}(3,3.5)(-1,-1)
\put(0,0){\line(1,0){2}} \put(0,0){\line(0,1){2}}
\put(0,2){\line(1,0){2}} \put(2,0){\line(0,1){2}}
\put(1,-.5){\makebox(0,0){1}}
\end{picture}
\begin{picture}(4,4.5)(-1,-1)
\put(0,0){\line(1,0){4}} \put(4,0){\line(0,1){4}}
\put(0,0){\line(1,1){4}} \put(2,-.5){\makebox(0,0){2}}
\end{picture}
} \qquad \qquad 
{ \setlength{\unitlength}{4mm}
\begin{picture}(5,4)(-1,-1)
\put(0,0){\line(1,0){4}} \put(4,0){\line(0,1){2}}
\put(0,0){\line(0,1){2}} \put(0,2){\line(1,0){4}}
\put(2,-.5){\makebox(0,0){2}} \put(-.5,1){\makebox(0,0){1}}
\end{picture}
\begin{picture}(4,4)(-1,-1)
\put(0,0){\line(1,0){3}} \put(3,0){\line(0,1){3}}
\put(0,0){\line(1,1){3}} \put(1.5,-.6){\makebox(0,0){$\sqrt{2}$}}
\end{picture}
} \caption{Chapman's isospectral two-piece bands $ A$ and $B$}
\label{fig:Chap2piece}
\end{figure}

\begin{theorem}\label{the1} The two-piece Chapman bands $A$ and
$B$ are isospectral but not isoheat. In particular, we have the
following expressions for their heat contents.

\begin{enumerate}
\item[\textup{(i)}]For $t\rightarrow \infty$,
\begin{equation}\label{e14}
Q_A(t)=\frac{1024}{9\pi^4}e^{-5\pi^2t/4}+O(e^{-2\pi^2t}).
\end{equation}

\item[\textup{(ii)}]For $t\rightarrow \infty$,

\begin{equation}\label{e15}
Q_B(t)=\frac{1152}{9\pi^4}e^{-5\pi^2t/4}+O(e^{-2\pi^2t}).
\end{equation}
\end{enumerate}
\end{theorem}

\begin{proof}
The eigenvalues for
the rectangle with vertices $(0,0),(a,0)$, $(0,b),(a,b)$ are
$\pi^2((\frac{n}{a})^2 + (\frac{m}{b})^2)$, where $n\in \N$ and
$m\in \N$. The corresponding eigenfunctions are
$\phi_{n,m}(x,y)=\sin (\frac{n \pi x}{a}) \sin (\frac{m \pi
y}{b})$. For a right-angled isosceles triangle with vertices
$(0,0),(c,0)$, and $(c,c)$, the eigenvalues are $\pi^2((\frac{i}{c})^2 +
(\frac{j}{c})^2)$, where $i\in \N$ and $j\in \N$ with $i
> j$. The corresponding eigenfunctions are $\psi_{i,j}(x,y)=\sin
(\frac{i \pi x}{c}) \sin (\frac{j \pi y}{c}) - \sin (\frac{j \pi
x}{c} )\sin (\frac{i \pi y}{c})$.
         We find that the first few eigenvalues of $ A$ and $ B$ are
$5 \pi^2/4$, $2 \pi^2, 5 \pi^2/2, 13\pi^2/4, 17\pi^2/4, \dots$.
Note that $5\pi^2/4$ is the lowest eigenvalue of the triangle in $
A$ and of the rectangle in $ B$. We find that for the rectangle
\begin{equation*}
\lVert\phi_{n,m}\rVert_2^2=ab/4,
\end{equation*}
and
\begin{equation*}
\int_{[0,a]\times[0,b]}
\phi_{n,m}(x,y)dxdy=\frac{ab}{mn\pi^2}(1+(-1)^{n-1})(1+(-1)^{m-1}).
\end{equation*}
It follows that with $a=1,b=2,m=n=1$, and recalling that $\phi_{1,
B}$ is normalized in $L^2( B)$,
\begin{equation*}
\left(\int _{ B}\phi_{1,
B}\right)^2=\frac{128}{\pi^4}=\frac{1152}{9\pi^4}.
\end{equation*}

It is convenient to consider the general case where $\lambda_1(M)$
has multiplicity $1$. By Bessel's inequality
\begin{align*}
\sum_{j=2}^{\infty}e^{-t\lambda_j(M)}\left(\int_M\phi_{j,M}(x)dx\right)^2&\le
e^{-t\lambda_2(M)}\sum_{j=2}^{\infty}\left(\int_M\phi_{j,M}(x)dx\right)^2
\nonumber \\ & \le e^{-t\lambda_2(M)}|M|.
\end{align*}
We conclude that if $\lambda_1(M)$ has multiplicity $1$ then
\begin{equation}\label{e21}
Q_M(t)=
e^{-t\lambda_1(M)}\left(\int_M\phi_{1,M}(x)dx\right)^2+O(e^{-t\lambda_2(M)}).
\end{equation}
This proves (ii) since $\lambda_1( B)=5\pi^2/4$, and $\lambda_2(
B)=2\pi^2$.

We find for the right-angled isosceles triangle with edges of
lengths $c, c$ and $c\sqrt 2$ respectively that
\begin{align}\label{e22}
\lVert\psi_{i,j}\rVert_2^2=&c^2\int_{[0,1]}dx\int_{[0,x]}dy((\sin(i\pi
x)\sin(j\pi y))^2+(\sin(j\pi x)\sin(i\pi y))^2\nonumber
\\ &\ \ -2\sin(i\pi x)\sin(j\pi y)\sin(j\pi x)\sin(i\pi y))=\frac{c^2}{4},
\end{align}
and
\begin{align}\label{e23}
\int_{[0,c]}&dx\int_{[0,x]}dy\psi_{2,1}(x,y)\nonumber
\\ &=c^2\int_{[0,1]}dx\int_{[0,x]}dy(\sin(2\pi x)\sin(\pi
y)-\sin(2\pi y)\sin(\pi x))\nonumber \\ &=-\frac{8c^2}{3\pi ^2}.
\end{align}
It follows by \eqref{e22} and \eqref{e23} for $c=2$ that
\begin{equation*}
\left(\int _{A}\phi_{1, A}\right)^2=\frac{1024}{9\pi^4}.
\end{equation*}
This proves (i) by \eqref{e21} since $\lambda_1( A)=5\pi^2/4$, and
$\lambda_2( A)=2\pi^2$.
\end{proof}

Using the fact that $ A$ and $ B$ are isospectral one can
construct a whole family of isospectral bands as follows. Add to $
A$ and $ B$ a disjoint square with area $1/2$, and now use the
fact that this square with the right angled isosceles triangle in
set $ B$ is isospectral to $2^{-1/2} B$, where $2^{-1/2}B$ means
that we scale the original polygons in $B$ by a factor of
$2^{-1/2}$. We proceed by induction to obtain two bands $C$ and
$D$ which are isospectral and not isoheat. $ C$ is a disjoint
union of squares with areas $2^{-j}:j \in \{0,1,2,\dots\}$
together with a right angled isosceles triangle of area $2$. $ D$
is a disjoint union of rectangles $R_0,R_1,\dots$, where $R_j$ has
area $2^{1-j}$ and diameter $2^{-j/2}\sqrt 5$.

\begin{corollary}\label{cor2} The infinite Chapman bands $
C$ and $ D$ are isospectral and not isoheat. In particular, we
have the following expressions for their heat contents.
\begin{enumerate}
\item[\textup{(i)}]For $t\rightarrow \infty$,
\begin{equation*}
Q_{ C}(t)=\frac{1024}{9\pi^4}e^{-5\pi^2t/4}+O(e^{-2\pi^2t}).
\end{equation*}

\item[\textup{(ii)}]For $t\rightarrow \infty$,

\begin{equation*}
Q_{ D}(t)=\frac{1152}{9\pi^4}e^{-5\pi^2t/4}+O(e^{-2\pi^2t}).
\end{equation*}
\end{enumerate}
\end{corollary}

\begin{proof}
Note that
$5\pi^2/4$ is the lowest eigenvalue of the triangle in $ C$ and of
the largest rectangle in $ D$. So $\left(\int _{
C}\phi_{1,C}\right)^2=\left(\int _{ A}\phi_{1,
A}\right)^2=\frac{1024}{9\pi^4}$ and $\left(\int _{ D}\phi_{1,
D}\right)^2=\left(\int _{ B}\phi_{1,
B}\right)^2=\frac{1152}{9\pi^4}$. Since $| C|=| D|<\infty$ we may
apply \eqref{e21} to complete the proof.
\end{proof}

 Given the ``fractal'' nature of $ C$ and
$ D$ it is of independent interest to obtain the first few terms
in the expansion for $t\downarrow 0$ of the heat trace $Z_{
C}(t)=Z_{ D}(t)$ as well as the heat contents $Q_{C}(t)$ and
$Q_{D}(t)$. We illustrate the techniques involved for the heat
trace in Theorem \ref{the2} and Corollary \ref{cor5}, and for the
heat content in Theorem \ref{the4}.

The original drums of Gordon et al. \cite{GWW}, the various
examples of planar isospectral drums subsequently constructed by
Buser et al. \cite{BCDS}, Chapman's modifications and the
isospectral examples with mixed boundary conditions of Jakobson et
al. \cite{JLNP} have a common source. Band, Parzanchevski and
Ben-Shach \cite{BPB,PB10} developed a generalization of Sunada's
isospectral construction \cite{S} which reproduces all the
examples above. In particular they prove that all examples
produced in terms of their method posses a transplantation. Later
on, Herbrich \cite{HT} showed the converse: domains which are
transplantable can also be constructed using the isospectral
theory in \cite{BPB,PB10}. Herbrich worked in the context of mixed
boundary conditions, and showed how to translate the
transplantability condition satisfied by the examples in
\cite{BCDS,C,GWW} into graph theory. This led to an algorithm for
finding transplantable pairs and ways of generating new pairs from
known ones. In particular, Herbrich finds twelve versions of the
drums of Gordon et al. of which ten have mixed boundary conditions
(see Fig. 5.6 in \cite{HT}). Imitating Chapman's modification of
the original drums of Gordon et al., we obtain the isospectral
three-piece bands with mixed boundary conditions in Figure
\ref{fig:brokenChap}.

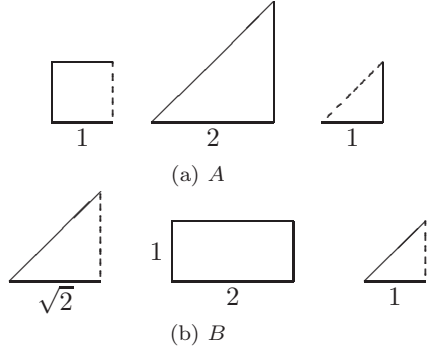
\begin{figure}
\centering \subfloat[$A$]{ \setlength{\unitlength}{4mm}
\begin{picture}(3,3)(-1,-1)
\put(0,0){\line(1,0){2}} \put(0,0){\line(0,1){2}}
\put(0,2){\line(1,0){2}} \curvedashes{0.2,0.2}
\put(0,0){\curve(2,0,2,0.5,2,1,2,2)} \put(1,-.5){\makebox(0,0){1}}
\end{picture}
\begin{picture}(4,4.5)(-1,-1)
\put(0,0){\line(1,0){4}} \put(4,0){\line(0,1){4}}
\put(0,0){\line(1,1){4}} \put(2,-.5){\makebox(0,0){2}}
\end{picture}
\hspace{0.4cm}
\begin{picture}(3,3)(-1,-1)
\put(0,0){\line(1,0){2}} \put(2,0){\line(0,1){2}}
\curvedashes{0.2,0.2} \put(0,0){\curve(0,0,.5,.5,1,1,2,2)}
\put(1,-.5){\makebox(0,0){1}}
\end{picture}
} \qquad \qquad \subfloat[$ B$]{ \setlength{\unitlength}{4mm}
\begin{picture}(4,4)(-1,-1)
\put(0,0){\line(1,0){3}} \curvedashes{0.2,0.2}
\put(0,0){\curve(3,0,3,1,3,2,3,3)} \put(0,0){\line(1,1){3}}
\put(1.5,-.6){\makebox(0,0){$\sqrt{2}$}}
\end{picture}
\hspace{.3cm}
\begin{picture}(5,4)(-1,-1)
\put(0,0){\line(1,0){4}} \put(4,0){\line(0,1){2}}
\put(0,0){\line(0,1){2}} \put(0,2){\line(1,0){4}}
\put(2,-.5){\makebox(0,0){2}} \put(-.5,1){\makebox(0,0){1}}
\end{picture}
\hspace{0.3cm}
\begin{picture}(3,3)(-1,-1)
\put(0,0){\line(1,0){2}} \curvedashes{0.2,0.2}
\put(0,0){\curve(2,0,2,0.5,2,1,2,2)} \put(0,0){\line(1,1){2}}
\put(1,-.5){\makebox(0,0){1}}
\end{picture}
} \caption{Isospectral three-piece bands with mixed boundary
conditions} \label{fig:brokenChap}
\end{figure}

\begin{corollary}\label{cor3}
Let $ A$ be the disjoint union of a square with area $1$, a
right-angled isosceles triangle with area $2$ and a right-angled
isosceles triangle with area $\frac12$. One edge of the square has
Neumann boundary conditions, and the $\sqrt 2$ edge in the
right-angled isosceles triangle also has Neumann boundary
conditions. Let $ B$ denote the disjoint union of a right
isosceles triangle with area $1$, the rectangle of edge lengths
$1$ and $2$, and the right isosceles triangle with area
$\frac{1}{2}$. One $\sqrt 2$ edge in the right-angled isosceles
triangle of area $1$ has Neumann boundary conditions, and one edge
in the right-angled isosceles triangle with area $\frac12$ also
has Neumann boundary conditions. All other edges in both $ A$ and
$ B$ have Dirichlet boundary conditions. Then $ A$ and $ B$ are
not isoheat. In particular, we have the following expressions for
their heat contents.
\begin{enumerate}
\item[\textup{(i)}]For $t\rightarrow \infty$,
\begin{equation}\label{e151}
Q_{ A}(t)=\frac{1600}{9\pi^4}e^{-5\pi^2t/4}+O(e^{-2\pi^2t}).
\end{equation}

\item[\textup{(ii)}]For $t\rightarrow \infty$,

\begin{equation}\label{e152}
Q_{ B}(t)=\frac{1664}{9\pi^4}e^{-5\pi^2t/4}+O(e^{-2\pi^2t}).
\end{equation}
\end{enumerate}
\end{corollary}

\begin{proof}
The large isosceles triangle in $ A$ has first eigenvalue
$\frac{5\pi^2}{4}$. By doubling the square across the Neumann edge
and recalling that the first Dirichlet eigenfunction on that
rectangle satisfies Neumann boundary conditions on the short axis
of symmetry we have that this square also has a first eigenvalue
$5\pi^2/4$. Furthermore, since the heat flow in the double of the
square is also symmetric with respect to the short axis, we have
by \eqref{e15} that this square contributes
$\frac{1}{2}\cdot\frac{1152}{9\pi^4}e^{-5\pi^2t/4}+O(e^{-2\pi^2t})$
to the heat content in $ A$. The contribution of the large
isosceles triangle is given in \eqref{e14}. Addition of these
contributions yields \eqref{e151}. Arguing in a similar fashion we
see that both the rectangle and the large isosceles triangle in
$B$ have a first eigenvalue $\frac{5\pi^2}{4}$. The contribution
of the isosceles triangle is
$\frac{1}{2}\cdot\frac{1024}{9\pi^4}e^{-5\pi^2t/4}+O(e^{-2\pi^2t})$,
whereas the contribution from the rectangle is given in
\eqref{e15}. Addition of these contributions yields \eqref{e152}.
\end{proof}

We note that the three-piece band $A$ has the square $E$ of
Example \ref{exa6} as a component, and that the three-piece band
$B$ has its isospectral partner $F$ as a component. Hence the two
piece bands $A-E$ and $B-F$ are isospectral. These bands both have
area $\frac52$, but the lengths of their Dirichlet boundaries
equal $6+2\sqrt 2$ and $7+\sqrt 2$ respectively. Hence these
two-piece bands are not isoheat.

These examples show that isospectral does not imply isoheat in the
setting of planar polygons.
  It would be interesting to have an isospectral non-isometric pair of planar polygons with the same heat content.


\medskip


\section{Heat content and Schr\"odinger operators} \label{sec3}

In this section, we return to the setting of Schr\"odinger
operators $L_q=-\frac{d^2}{dx^2}+q$ acting in $L^2[0,1]$ with
Dirichlet boundary conditions at $0$ and at $1$. We assume that
$q\in L^2[0,1]$ and give an abundance of isospectral deformations
that are not isoheat. For this purpose consider the vector fields
$X_n(q)=2\frac{d}{dx}(\phi_{n,q})^2, n\in \N$, introduced and
studied in great detail in \cite{PT}. Denote by $s\mapsto
\gamma_n(s)$ the solution of
$\frac{d}{ds}\gamma_n(s)=X(\gamma_n(s))$ with $\gamma_n(0)=0$. By
\cite[p. 91]{PT}
\begin{equation}\label{e155}
\gamma_n(s)(x)\equiv \gamma_n(x;s)=-2\partial_x^2\log
\theta_n(x;s),
\end{equation}
where
\begin{equation}\label{e156}
\theta_n(x;s)=1+(e^s-1)\int_x^12(\sin(n\pi r))^2dr.
\end{equation}
Note that
\begin{equation}\label{e157}
\theta_n(x;s)=\int_0^x2(\sin(n\pi r))^2dr+e^s\int_x^12(\sin(n\pi
r))^2dr,
\end{equation}
and hence
\begin{equation}\label{e158}
\min\{1,e^s\}\le \theta_n(x;s)\le\max\{1,e^s\}, s\in \R.
\end{equation}
Furthermore an explicit computation shows that
\begin{equation}\label{e159}
\theta_n(x;s)=1+(e^s-1)\left(1-x+\frac{\sin(2n\pi
x)}{2n\pi}\right).
\end{equation}
In particular $\gamma_n(x;s)$ is smooth in $(x,s)\in \R\times \R
$. It has been shown in \cite{PT} that for any $n\in \N$ and $s\in
\R$, $\gamma_n(s)$ and the $0$ potential are isospectral. Hence
$\gamma_n(s)$ is an isospectral deformation.

\begin{theorem}\label{the3}
For any $n\in \N, \gamma_n(s)$ is not an isoheat deformation.
\end{theorem}

The proof of Theorem \ref{the3} relies on an explicit computation
of the quantities $\left( \int ^1_0 \phi _{j, \gamma _n(s)} (x) dx
\right) ^2$ near $s = 0$, based on the formulas for $\gamma _n
(s)$ and $\phi _{j,\gamma _n(s)}$ in \cite{PT}.
\medskip

\begin{proof}
Recall that (i) for the Schr\"odinger operator $L_q$ with
$q = 0$, an orthonormal basis of eigenfunctions in $L^2 [0,1]$ is
given by
 $\phi _{j,0}(x) = \sqrt{2} \sin (j\pi x), \ j \geq 1$, and (ii) the
isospectral deformation $\gamma _n(s)$ with initial condition
$\gamma_n(0)=0$ is given by \eqref{e155} and \eqref{e156}.

For $j \not= n$, the eigenfunction $\phi _{j,\gamma _n(s)}$ is
computed in \cite[p. 92]{PT} as
\begin{equation*}
\phi _{j,\gamma _n(s)} (x) = \sqrt{2} \sin (j\pi x) - (e^s - 1)
\frac{\sqrt{2} \sin (n\pi x)}{\theta _n(x;s)} \int ^1_x 2 \sin
(j\pi r) \sin (n\pi r)dr,
\end{equation*}
whereas for $j = n$,
\begin{equation*}
\phi _{n,\gamma
_n(s)} (x) = e^{s/2} \frac{\sqrt{2} \sin (n\pi x)}{\theta
_n(x;s)}.
\end{equation*}
These formulae are now used to expand the coefficients
$h_j(s)\equiv h_{n,j}(s) = \int ^1_0 \phi _{j,\gamma _n(s)} (x)dx$
in the heat content $Q_{\gamma _n(s)}(t) = \sum _{j\ge 1} e^{-t
j^2 \pi ^2} h^2_j(s)$ at $s = 0$. Note that
   \[ h_j(0) = \sqrt{2} \int ^1_0 \sin (j \pi x)dx = \frac{\sqrt{2}}{j\pi }
      \left( 1 - (-1)^j \right)
   \]
and $\theta _n(x; 0) = 1$ whereas
   \[ \partial _s \big\arrowvert _{s=0} \theta _n(x;s) = \int ^1_x 2 (\sin (n
      \pi r))^2dr .
   \]
Using that
   \[ \partial _s \big\arrowvert _{s=0} \phi _{n,\gamma _n(s)} (x) = \frac{1}
      {\sqrt{2}} \sin (n\pi x) - \sqrt{2} \sin (n\pi x) \partial _s \big
      \arrowvert _{s = 0} \theta _n(x; s)
   \]
one gets
   \[ \partial _s \big\arrowvert _{s=0} h_n(s) = - \frac{\sqrt{2}}{n\pi }
      \frac{1 + (-1)^n}{2} .
   \]
Similarly, for $j \not= n$ one has
   \[ \partial _s \big\arrowvert _{s=0} \phi _{j,\gamma _n(s)} (x) = -\sqrt{2}
      \sin (n\pi x) \int ^1_x 2 \sin (j\pi r) \sin (n\pi r) dr
   \]
leading to
   \[ \partial _s \big\arrowvert _{s=0} h_j(s) = \frac{1}{\sqrt{2}} \frac{1}
      {n\pi } \delta _{j,2n}
   \]
where $\delta _{j,2n}$ is the Kronecker delta. Hence for $j \notin
\{ 2n, n \}$,
   \[ h^2_j(s)  = \begin{cases} O(s^4), &j \mbox{ even}   \\
      \frac{8}{j^2\pi ^2} + O(s^2), &j \mbox{ odd} \end{cases}
   \]
whereas $h^2_{2n}(s) = \frac{1}{2n^2 \pi ^2} s^2 + O(s^3)$ and
   \[  h^2_n(s)  = \begin{cases} \frac{2}{n^2\pi ^2} s^2 +
       O(s^3), &n \mbox{ even}   \\
      \frac{8}{n^2\pi ^2} + O(s^2), &n \mbox{ odd}  \ . \end{cases}
   \]
As a consequence $Q_{\gamma _n(s)}(t)$ is not an isoheat
deformation for $s$ near $0$.
\end{proof}

\medskip

The expansions of the coefficients $h_{n,j}(s)$ at $s=0$ in the
above proof show that for any $j \geq 1$ and $n \ge 1$
   \[ \partial _s \big\arrowvert _{s=0}  h^2_{n,j}(s) = 0,
   \]
leading to the following.

\medskip

\begin{corollary} \label{cor4}
For any $n \geq 1$
   \[ \partial _s \big\arrowvert _{s=0} Q_{\gamma _n(s)}(t) = 0, \quad
      \forall t > 0 .
   \]
\end{corollary}

\medskip

Formally, Corollary \ref{cor4} means that for any $t> 0$, the
differential $d_0Q(t)$ of the map $Q_.(t): q \to Q_q(t)$ is not
1-1. In fact, the subspace of $L^2[0,1],$ spanned by the elements
$X_n(0),$ $n\ge 1,$ is in the kernel of $d_0Q(t)$. By setting up
the maps $Q_.(t)$ in the appropriate spaces, these statements
could be made precise.

We briefly comment on a second family of vector fields, $Y_n(q), n \geq 1$,
introduced in \cite[p. 108]{PT},
   \[ Y_n(q) = - 2 \partial _x(a_n - [a_n] \phi ^2 _{n,q})
   \]
where
   \[ a_n(x,q) = y_1(x,\lambda _n(q),q) y_2(x, \lambda _n(q), q)
   \]
and $[a_n] = \int ^1_0 a_n(x,q)dx$. Here $y_1(x, \lambda ,q)$ and
$y_2(x, \lambda ,q)$ denote the fundamental solutions of $-y'' +
qy = \lambda y,$ satisfying the initial conditions
$y_1(0,\lambda,q) = 1, y_1'(0,\lambda,q)=0$, and $y_2(0,\lambda,q)
= 0, y_2'(0,\lambda,q)=1$ respectively. We see \cite[p. 65]{PT}
that $Y_n(q) \in E_0$ for any $q \in E_0$ and $n \geq 1$ where
$E_0$ denotes the subspace of even potentials, $q(1 - x) = q(x)$
for all $x \in [0,1]$, satisfying $[q] = 0$. Actually, by the
construction of $Y_n(q)$, at any $q \in E_0$, the closure of the
span of $Y_n (q), n \geq 1$, is $E_0$ \cite[p. 107]{PT}. By
\cite[p. 111]{PT}, the initial value problem
   \[ \frac{d}{ds} \xi (s) = Y_n(\xi (s)) , \ \xi (0) = 0
   \]
has a unique solution $\xi _n(s)$ which exists for $s \in {\mathbb
R}$ satisfying
   \[ (n - 1)^2 \pi ^2 < n^2 \pi ^2 + s < (n + 1)^2 \pi ^2 .
   \]
The main feature of the flow $\xi _n(s)$ is that
   \[ \lambda _{n,j}(s) \equiv \lambda _j(\xi _n(s)) = j^2 \pi ^2 + s
      \delta _{nj}
   \]
where $\delta _{nj}$ denotes the Kronecker delta -- see \cite[p.
108]{PT}.

As for the eigenfunctions, recall that the eigenfunctions
corresponding to $q_*=q(1-x)$, denoted $\phi_{j,q_*}(x)$, are
related to the eigenfunctions corresponding to $q$ by
$\phi_{j,q_*}(x)=(-1)^{j+1}\phi_{j,q}(x)$ \cite[p. 42]{PT}. As a
consequence,
$\int_0^1\phi_{j,q_*}(x)dx=(-1)^{j+1}\int_0^1\phi_{j,q}(x)dx$.  It
follows that for $q$ even, i.e., $q=q_*$, $\phi_{j,q}(x)$ is even
(odd) if $j$ is odd (even). Hence for $j$ and $q$ both even,
$\int_0^1\phi_{j,q}(x)dx=0$.  Thus, as $\xi _n(s) \in E$, we have
   \[ \int ^1_0 \phi _{j,\xi _n(s)}(x) dx = 0 \quad \forall j \in \{ 2k
      \big\arrowvert k \geq 1 \} .
   \]
Therefore for $n$ even, the heat content $Q_{\xi _n(s)}(t)$ is
given by
\[Q_{\xi _n(s)}(t) = \sum _{j \, {\rm odd}} e^{-j^2\pi ^2 t}
                \left( \int ^1_0 \phi _{j,\xi _n(s)} (x) dx \right) ^2.
\]
Given the fact that the flow $\xi _n(s)$ leaves any Dirichlet eigenvalue
$\lambda _j(s)$ with $j \not= n$ invariant one could be tempted to believe
that $(\int ^1_0 \phi _{j,\xi _n(s)} (x)dx)^2$ is independent of $s$.
However, this is not the case. Using the formula for $y_2(x, j^2 \pi ^2,
\xi _n(s))$ given in \cite[p. 111]{PT} one can prove that for $j \mbox{ odd}$
   \[ \partial _s \big\arrowvert _{s = 0} \left( \int ^1_0 \phi _{j,
      \xi _n(s)} (x) dx \right) ^2 \not= 0 .
   \]
In fact, by a tedious but straightforward computation, one sees
that for $n$ even
   \[ \partial _s \big\arrowvert _{s=0} Q_{\xi _n(s)} (t) = 8 \sum _{j \,
      {\rm odd}} e^{-j^2 \pi ^2 t} \frac{1}{j^2 \pi ^2} \left( \frac{1}
      {(j^2-n^2)\pi ^2} + \frac{1}{n^2 \pi ^2} \right),
   \]
implying that $\partial _s \big\arrowvert _{s=0} Q_{\xi _n(s)}
\not= 0$. A similar conclusion holds in the case where $n$ is odd.
However, the additional summand $j=n$ is more complicated, leading
to the formula
\[ \partial _s \big\arrowvert _{s=0} Q_{\xi _n(s)} (t) = 8 \sum _{j \,
      {\rm odd}, \ j \ne n} e^{-j^2 \pi ^2 t} \frac{1}{j^2 \pi ^2} \left( \frac{1}
      {(j^2-n^2)\pi ^2} + \frac{1}{n^2 \pi ^2} \right)
\]
\[
    -\frac{8t}{n^2\pi^2} e^{-n^2 \pi ^2 t}
       + 8e^{-n^2 \pi ^2 t} \int ^1_0 \partial _s \big\arrowvert _{s=0} y_2(x, n^2\pi^2 +s ,\xi _n(s)) dx,
\]
where the last term can be explicitly computed using \cite[p.
111]{PT}.

\section{Appendix: Heat trace and heat content in the fractal setting} \label{sec4}

Below we shall compute the first few terms of the asymptotic
expansion for the heat trace of the band $ C$ as constructed in
Section \ref{secplanarexas}. Since $ C$ is isospectral to $ D$,
and since $D$ consists of disjoint rectangles only, it is easier
to consider the latter. $Z_{ D}(t)$ can be written as a triple sum
over $\N^3$. However, it is easier to use the approach via the
renewal equation \cite{LV}. This also allows us to consider
general disjoint unions of self-similar polygons. Our {\red
set-up} is the following.

Let $P$ be a polygon with area $|P|$ and boundary length
$|\partial P|$. We abbreviate
\begin{equation*}
V(P)=\sum_{i=1}^n\frac{\pi^2-\gamma_i^2}{24\pi \gamma_i},
\end{equation*}
and
\begin{equation*}
R(t)=Z_P(t)-\frac{|P|}{4\pi t}+\frac{|\partial P|}{8(\pi
t)^{1/2}}-V(P),
\end{equation*}
so that \eqref{e11} takes the form
\begin{equation}\label{e36}
|R(t)|\le d_1(P)e^{-d_2(P)/t},\ t>0.
\end{equation}

Let $0<\alpha <1$ and denote by $\alpha P$ the rescaled polygon
$\{\alpha x:x\in P\}$. We let $ P_{\alpha}$ be a disjoint union of
sets $\alpha^j P: j \in \N\cup\{0\}$. It is easily seen that
\begin{equation*}
|P_{\alpha}|=(1-\alpha^2)^{-1}|P|, \ \ |\partial
P_{\alpha}|=(1-\alpha)^{-1}|\partial P|.
\end{equation*}
Our main result is the following asymptotic expansion for the heat
trace of $P_{\alpha}$.
\begin{theorem}\label{the2} There exists a $2\log(1/{\alpha})$ periodic function $\pi_{\alpha,P}$ such that for $t\downarrow 0$,
\begin{equation}\label{e28}
Z_{ P_{\alpha}}(t)=\frac{| P_{\alpha}|}{4\pi t}-\frac{|\partial
 P_{\alpha}|}{8(\pi t)^{1/2}}+c_{\alpha,P}\log t+\pi_{\alpha,P}(\log t) +O(e^{-d_2(P)/(\alpha^2t)}),
\end{equation}
where $d_2(P)$ is as in \eqref{e11} and
\begin{equation}\label{e29}
c_{\alpha,P}=\frac{V(P)}{2\log \alpha}.
\end{equation}
\end{theorem}

\begin{corollary}\label{cor5}
Let $ C$ and $ D$ be the bands constructed in Section
\ref{secplanarexas}. Then for $t\downarrow 0$,
\begin{equation}\label{e30}
Z_{ C}(t)=Z_{ D}(t)=\frac{1}{\pi t}-\frac{3(2+\sqrt 2)}{4(\pi
t)^{1/2}}-\frac{\log t}{4\log 2}+\pi_{1/{\sqrt 2},P}(\log t)
+O(e^{-2d_2(P)/t}),
\end{equation}
where $P$ is a rectangle with edges of lengths $1$ and $2$.
\end{corollary}
\noindent{\it Proof of Corollary \ref{cor5}.} Note that $|P|=2,
|\partial P|=6, V(P)=1/4$, $\alpha=1/{\sqrt{2}}$, and $D=
P_{1/\sqrt2}$. Formula \eqref{e30} follows from Theorem \ref{the2}
and the fact that $C$ and $D$ are isospectral.\hspace*{\fill
}$\square $

\noindent{\it Proof of Theorem \ref{the2}.} By scaling we have
that for $\alpha>0, t>0$,
\begin{equation*}
Z_{\alpha P}(t)=Z_P(t/{\alpha^2}).
\end{equation*}
Since the heat trace is additive on a union of disjoint sets we
have that
\begin{align}\label{e32}
Z_{ P_{\alpha}}(t)&=Z_{\alpha
P_{\alpha}}(t)+Z_P(t)\nonumber \\
&=Z_{ P_{\alpha}}(t/{\alpha^2})+Z_P(t).
\end{align}
If we substitute
\begin{equation*}
Z_{ P_{\alpha}}(s)=\frac{| P_{\alpha}|}{4\pi s}-\frac{|\partial
 P_{\alpha}|}{8(\pi s)^{1/2}}+c_{\alpha,P}\log s+U(s),
\end{equation*}
for the relevant expressions in \eqref{e32} then we obtain
\begin{equation}\label{e34}
U(t)-U(t/{\alpha^2})=R(t).
\end{equation}

Equation \eqref{e34} is of renewal type, and has been studied
extensively in relation to fractal geometry. See \cite{LV} and the
references therein. However, in order to invoke the Renewal
Theorem \cite[p. 198]{LV}, two-sided decay estimates on $R(t)$ are
required. From \eqref{e36} we can deduce decay for $t\downarrow
0$, but no decay for $t\rightarrow \infty$. Indeed, for
$t\rightarrow \infty$ all of $Z_P(t), \frac{| P|}{4\pi t}$, and
$\frac{|\partial
 P|}{8(\pi t)^{1/2}}$ decay to $0$. But the angle contribution $V(P)$ is $t$-independent and so the remainder does not decay for $t\rightarrow \infty$.

Iteration of \eqref{e34} yields for any $j \in \N$,
\begin{equation*}
U(\alpha^{2j}t)=U(t/{\alpha^2})+\sum_{i=0}^jR(\alpha^{2i}t).
\end{equation*}
Substitution of $t=\alpha^2 \theta$ yields
\begin{equation*}
U(\alpha^{2j+2}\theta)=U(\theta)+\sum_{i=0}^jR(\alpha^{2i+2}\theta).
\end{equation*}
We define $j(t)\in \Z$ by
\begin{equation*}
\alpha^{2j(t)+2}\le t< \alpha^{2j(t)},
\end{equation*}
and $\theta(t)\in [1,\alpha^{-2})$ by
\begin{equation*}
\theta(t)=\alpha^{-2j(t)-2}t.
\end{equation*}
Then for $t<1$, $j(t)\ge 0$ and
\begin{equation*}
U(t)=U(\theta(t))+\sum_{i=0}^{\infty}R(\alpha^{2i+2}\theta(t))-\sum_{i=j(t)+1}^{\infty}R(\alpha^{2i+2}\theta(t)).
\end{equation*}
To complete the proof we have that for $t<1$,
\begin{align*}
\bigg\lvert\sum_{i=j(t)+1}^{\infty}R(\alpha^{2i+2}\theta(t))\bigg\rvert&
\le \sum_{i=j(t)+1}^{\infty}\lvert R(\alpha^{2i+2}\theta(t))\rvert
=\sum_{i=1}^{\infty}\lvert R(\alpha^{2i}t)\rvert\nonumber \\ & \le
d_1(P)\sum_{i=1}^{\infty}e^{-d_2(P)/(\alpha^{2i}t)}\nonumber \\ &
=O(e^{-d_2(P)/(\alpha^2t)}).
\end{align*}
Finally we note that $z\mapsto \theta(e^z)$ is a
$2\log(1/{\alpha})$ periodic function. Hence $U(t)$ is
$2\log(1/{\alpha})$-periodic in $\log t$. This completes the
proof.\hspace*{\fill }$\square $ \medskip

The corresponding result for the heat content is the following.
\begin{theorem}\label{the4} There exists a $2\log(1/{\alpha})$ periodic function $\phi_{\alpha,P}$ such that for $t\downarrow 0$,
\begin{equation*}
Q_{P_{\alpha}}(t)=|P_{\alpha}|-2\pi^{-1/2}|\partial
P_{\alpha}|t^{1/2}+d_{\alpha,P}\cdot t\cdot\log
t+\phi_{\alpha,P}\cdot t(\log t)
+O(t^{-1}e^{-d_4(P)/(\alpha^2t)}),
\end{equation*}
where $d_4(P)$ is as in \eqref{e12} and
\begin{equation}\label{e45}
d_{\alpha,P}=\frac{\sum_{i=1}^nc(\gamma_i)}{2\log \alpha}.
\end{equation}
\end{theorem}

\begin{proof}
Let
\begin{equation}\label{e46}S(t)=Q_P(t)-|P|+2\pi^{-1/2}|\partial
P|t^{1/2}-\sum_{i=1}^nc(\gamma_i)t.
\end{equation}
Then \eqref{e12} takes the form
\begin{equation*}
|S(t)|\le d_3(P)e^{-d_4(P)/t},\ t\ge 0.
\end{equation*}
By scaling we have that for $\alpha>0, t>0$,
\begin{equation*}
Q_{\alpha P}(t)=\alpha^2Q_P(t/{\alpha^2}).
\end{equation*}
Since the heat content is additive on a disjoint union of sets we
have that
\begin{align}\label{e49}
Q_{P_{\alpha}}(t)&=Q_{\alpha P_{\alpha}}(t)+Q_P(t)\\ \nonumber
&=\alpha^2Q_{P_{\alpha}}(t/{\alpha^2})+Q_P(t).
\end{align}
If we substitute
\begin{equation*}
Q_{P_{\alpha}}(s)=|P_{\alpha}|-2\pi^{-1/2}|\partial
P_{\alpha}|s^{1/2}+d_{\alpha,P}s\log s+sT(s),
\end{equation*}
in \eqref{e49}, and subsequently use \eqref{e45} and \eqref{e46}
then we obtain that for all $t>0$,
\begin{equation}\label{e50}
T(t)-T(t/{\alpha^2})=t^{-1}S(t).
\end{equation}
The remaining part of the proof is similar to the lines of the
proof of Theorem \ref{the2} below \eqref{e34}.
\end{proof}


\begin{thebibliography} {99}

\bibitem{BPB}R. Band, O. Parzanchevski, G. Ben-Shach \emph{The
isospectral fruits of representation theory: Quantum graphs and
drums}, J. Phys. A: Math. Theor. \textbf{42}, 175202 (2009).

\bibitem{vdB1}M. van den Berg, P. Gilkey \emph{Heat content asymptotics of a Riemannian manifold with boundary},
J. Funct. Anal. \textbf{120}, 48--71 (1994).

\bibitem{vdBGKK}M. van den Berg, P. Gilkey, K. Kirsten, V. A. Kozlov, \emph{Heat content asymptotics for Riemannian manifolds with
Zaremba boundary conditions}, Potential Analysis \textbf{26}, 225--254 (2007).

\bibitem{vdB2}M. van den Berg, S. Srisatkunarajah, \emph{Heat equation for a region in $\R^2$  with a polygonal boundary},
 J. Lond. Math. Soc. (2) \textbf{37} 119--127 (1988).

\bibitem{vdB3}M. van den Berg, S. Srisatkunarajah, \emph{Heat flow and Brownian motion for a region in $\R^2$  with a polygonal boundary},
 Probab. Theor. Rel. Fields \textbf{86}, 41--52 (1990).

\bibitem{vdBLG}M. van den Berg, J.- F. Le Gall, \emph{Mean curvature and the heat
equation}, Math. Z. \textbf{215}, 437--464 (1994).

\bibitem{BCDS}P. Buser, J. Conway, P. Doyle, K.-D. Semmler,
\emph{Some planar isospectral domains}, Intern. Math. Res. Notices
\textbf{9} 391--400 (1994).

\bibitem{C}S. J. Chapman, \emph{Drums that sound the same}, Amer. Math. Monthly \textbf{102}, 124--138 (1995).

\bibitem{PG}P. Gilkey, \emph{Asymptotic Formulae in Spectral
Geometry}, Stud. Adv. Math., Chapman\& Hall/CRC, Boca Raton, FL
(2004).

\bibitem{PG2}P. Gilkey, \emph{Heat content, heat trace and isospectrality}, New developments in Lie theory and geometry,
 Contemp. Math. \textbf{491}, 115--123 (2009).

\bibitem{GPS}C. Gordon, P. Perry, D. Schueth, \emph{Isospectral and isoscattering manifolds: a survey of techniques and examples},
Contemp. Math. \textbf{387}, 157--179 (2005).

\bibitem{GWW}C. Gordon, D. Webb, S. Wolpert, \emph{Isospectral plane domains via Riemannian
orbifolds}, Invent. Math.  \textbf{110}, 1--22 (1992).

\bibitem{HT}P. Herbrich, \emph{On inaudible properties of broken drums -- Isospectral domains with mixed boundary conditions},
http://arXiv.org/abs/1111.6789v2 (2012).

\bibitem{JLNP}D. Jakobson, M. Levitin, N. Nadirashvili, I.
Polterovich, \emph{Spectral problems with mixed Dirichlet-Neumann
boundary conditions: isospectrality and beyond}, J. Comput. Appl.
Math. \textbf{194}, 141--155 (2006).

\bibitem{MMgraphs}P. McDonald, R. Meyers, \emph{Isospectral polygons, planar graphs and heat content},
Proc. Amer. Math. Soc. \textbf{131}, 3589--3599 (2003).

\bibitem{MKS}H. P. McKean, I. M. Singer, \emph{Curvature and the
eigenvalues of the Laplacian}, J. Diff. Geom. \textbf{1}, 43--69
(1967).

\bibitem{LPP}M. Levitin, L. Parnovski, I. Polterovich,
\emph{Isospectral drums with mixed boundary conditions}, J. Phys.
A: Math. Gen. \textbf{39}, 2073--2082 (2006).

\bibitem{LV}M. Levitin, D. Vassiliev, \emph{Spectral asymptotics, renewal theorem, and the Berry conjecture for a class of
fractals}, Proc. Lond. Math. Soc. \textbf{72}, 188--214 (1996).

\bibitem{PB10}O. Parzanchevski, R. Band, \emph{Linear representations and isospectrality with boundary
conditions}, J. Geom. Anal. \textbf{20}, 439--471 (2010).

\bibitem{PT}J. P\"oschel, E. Trubowitz, \emph{Inverse spectral theory}, Pure and Applied Mathematics \textbf{130},
 Academic Press, Inc., Boston, MA, (1987).

\bibitem{S}T. Sunada, \emph{Riemannian coverings and isospectral
manifolds}, Ann. Math. \textbf{121}, 169--186 (1985).

\end{thebibliography}
\end{document}